\documentclass[10pt,reqno]{tran-l}
\copyrightinfo{2005}{American Mathematical Society}

\usepackage{amsmath,amsthm,amscd,amssymb}
\usepackage{amsfonts, amsmath}
\usepackage{graphicx}
\usepackage[usenames]{color}

\newtheorem{theorem}{Theorem}[section]

\newtheorem{corollary}[theorem]{Corollary}
\newtheorem{lemma}[theorem]{Lemma}

\numberwithin{equation}{section}

\def\Q{\nabla}
\def\D{\Delta}

\def\la{\lambda}
\def\O{\Omega}
\def\Om{\Omega}
\def\xx{{\bf x}}
\def\zz{{\bf z}}
\newcommand{\Ess}{\textup{ess sup}}
\newcommand{\ctilde}{\widetilde{C}_n}
\newcommand{\rz}{\mathbb{R}}

\begin{document}
\title[New Weyl-Type Bounds]
{Two New Weyl-Type Bounds for the Dirichlet Laplacian}
\author[Hermi]{Lotfi Hermi}
\address{Department of Mathematics,
University of Arizona, 617 Santa Rita, Tucson, AZ 85721 USA }
\email{hermi@math.arizona.edu}
\urladdr{http://www.math.arizona.edu/$\sim$hermi}

\commby{}

\date{April 15, 2004 and, in revised form, May 31, 2006.}


\begin{abstract}

In this paper, we prove two new Weyl-type upper estimates for the
eigenvalues of the Dirichlet Laplacian. As a consequence, we
obtain the following {\em lower} bounds for its counting function. For $\la\ge
\la_1$, one has
\begin{equation}
N(\la) > \dfrac{2}{n+2} \ \dfrac{1}{H_n} \
\left(\la-\la_1\right)^{n/2} \ \la_1^{-n/2}, \notag
\end{equation}
and
\begin{equation}
N(\la) > \left(\dfrac{n+2}{n+4}\right)^{n/2} \ \dfrac{1}{H_n} \
\left(\la-(1+4/n) \ \la_1\right)^{n/2} \ \la_1^{-n/2}, \notag
\end{equation}
where
\begin{equation}
H_n=\dfrac{2 \ n}{j_{n/2-1,1}^2 J_{n/2}^2(j_{n/2-1,1})} \notag
\end{equation}
is a constant which depends on $n$, the dimension of the
underlying space, and Bessel functions and their zeros.
\end{abstract}

\maketitle

\tableofcontents

\newpage

\section{Four New Estimates}
\label{introduction}

Let $\Om \subset \rz^n$ be a bounded domain with piecewise smooth
boundary. We are interested in bounds for the eigenvalues of the
fixed and free membrane whose shape is assumed by $\Om$. The first
problem (also called the Dirichlet problem) is described by the
equation,
\begin{eqnarray}
 - \D u &=& \la \ u  \, \text{ in } \O,  \\
u &=& 0 \, \quad \text{   on } \partial\O. \notag
\end{eqnarray}
Its eigenvalues, known to form a discrete countable family with no
finite accumulation point (see \cite{Ben1} \cite{CH} for example),
are denoted (counting multiplicity) by $0 <\la_1 <\la_2 \le \la_3
\le \ldots \le \la_k \to \infty.$ Its associated eigenfunctions,
which form an orthonormal basis of {\em real} functions in
$L^2(\Om)$, are denoted by $u_1, u_2, u_3, \ldots$. The second
problem (also called the Neumann problem \cite{Ben2}) is described
by
\begin{eqnarray}
 - \D v &=& \mu \ v  \, \text{ in } \O,  \\
\dfrac{\partial{v}}{\partial{n}} &=& 0 \, \quad \text{   on }
\partial\O. \notag
\end{eqnarray}
Its eigenvalues, also discrete and countable, are denoted by $0 =\mu_1
<\mu_2 \le \mu_3 \le \ldots \le \mu_k \to \infty.$ In this paper,
we show the following.
\begin{theorem} \label{theo1}
For $k\ge1$, we have
\begin{equation}
\la_{k+1}-\la_1 \le \left(1+\dfrac{n}{2}\right)^{2/n} H_n^{2/n} \,
\la_1  \, k^{2/n} \label{new1}
\end{equation}
and the sharper, average-type, inequality
\begin{equation}
\sum_{j=1}^{k} \left(\la_{j}-\la_1\right) \le \dfrac{n}{n+2}
H_n^{2/n} \, \la_1  \, k^{1+2/n} \label{new2}
\end{equation}
where
\begin{equation} \label{constant}
H_n=\dfrac{2 \ n}{j_{n/2-1,1}^2 J_{n/2}^2(j_{n/2-1,1})}.\end{equation}
\end{theorem}
As corollaries to these two inequalities, we prove (\ref{new3})
and (\ref{new4}).
Here $J_n(x)$ and $j_{n,p}$ denote, respectively, the Bessel
function of order $n$, and the $p$th zero of this function (see
\cite{AS}). The proof of this theorem is offered in Section
\ref{proof-of-main}. That (\ref{new2}) is sharper than
(\ref{new1}) follows from left Riemann sum considerations (See
Fig. 1), namely
\begin{equation} \label{left-riemann}
\sum_{j=0}^{k-1} \ j^{2/n} < \dfrac{k^{1+2/n}}{1+2/n}.
\end{equation}

\begin{figure}[htb!]
  \begin{center}
    \includegraphics{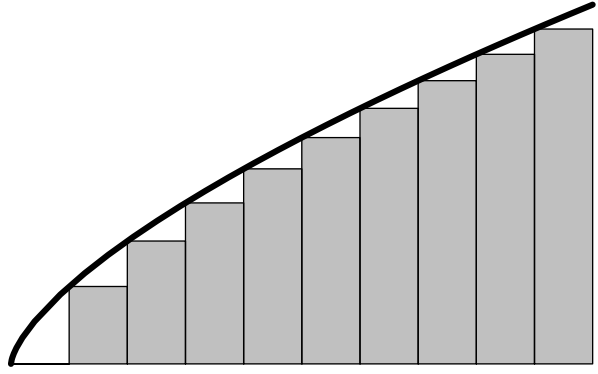}
    \caption{Left Riemann  Sums for $y=x^{2/n}$, for $0\le x\le k$.}
    \label{fig1}
  \end{center}
\end{figure}

\vskip 1 cm

To see this, we apply (\ref{new1}) to $j=0,1, \cdots, k-1$, then
sum. We obtain
\begin{eqnarray*}
\sum_{j=1}^{k} (\la_j- \la_1) &\le&
\left(1+\dfrac{n}{2}\right)^{2/n} H_n^{2/n} \ \la_1 \
\left(\sum_{j=0}^{k-1} j^{2/n}\right) \\ \notag
&<&\left(1+\dfrac{n}{2}\right)^{2/n} H_n^{2/n} \ \la_1 \
\dfrac{k^{1+2/n}}{1+2/n}.
\end{eqnarray*}
Inequality (\ref{new2}) is tighter since
\[\left(1+\dfrac{n}{2}\right)^{2/n} >1.\]

In fact, we have the following corollary.
\begin{corollary} \label{Chiti-Ashbaugh-Benguria}
For $k\ge 1$, the eigenvalues of the Dirichlet problem satisfy the
estimate
\begin{equation}
\sum_{j=2}^{k+1} \dfrac{1}{\la_{j}-\la_1} \ge \dfrac{n+2}{n} \
H_n^{-2/n} \ \dfrac{1}{\la_1} \ \dfrac{k^2}{(k+1)^{1+2/n}}.
\label{new3}
\end{equation}
\end{corollary}
This of course follows from the Cauchy-Schwarz inequality
\begin{align*}
k^2 &= \left( \sum_{j=2}^{k+1} \dfrac{1}{\sqrt{\la_j-\la_1}} \
\sqrt{\la_j-\la_1} \right)^2 \\ \notag &\le \sum_{j=2}^{k+1}
\dfrac{1}{\la_j-\la_1} \ \sum_{j=2}^{k+1} (\la_j-\la_1)
\end{align*}
and inequality (\ref{new2}). Another consequence of (\ref{new2})
and H.~C.~Yang's \cite{Y} (see also \cite{A1} \cite{A2},
\cite{AH2}, \cite{AH3}, \cite{AH4}) inequality
\[\la_{k+1} \le \left(1+\frac{4}{n}\right) \, \dfrac{1}{k} \ \sum_{j=1}^{k} \la_j, \]
is the bound
\begin{equation}
\la_{k+1} \le \left(1+\frac{4}{n}\right) \, \left(1+
\dfrac{n}{n+2} H_n^{2/n} \ k^{2/n}\right) \ \la_1. \label{new4}
\end{equation}

These types of inequalities follow the spirit of Weyl's asymptotic law, which states that
\begin{equation}
\la_k \sim \dfrac{4 \pi^2 k^{2/n}}{\left(C_n |\Om|\right)^{2/n}}
\text{ as } k \rightarrow \infty, \notag
\end{equation}
and
\begin{equation}
\mu_{k+1} \sim \dfrac{4 \pi^2 k^{2/n}}{\left(C_n
|\Om|\right)^{2/n}} \text{ as } k \rightarrow \infty, \notag
\end{equation}
where $C_n=\dfrac{\pi^{n/2}}{\Gamma(n/2+1)}=$ volume of the unit
$n-$ball, and $|\Omega|$ denotes the volume of $\Omega$.
These formulas were proved by Weyl \cite{Weyl} in 1910.

There is a beautiful expos\'e of the history of this problem in
Kac's paper \cite{Kac} (see also the equally entertaining paper
\cite{Protter}). Baltes and Hilf \cite{BH} trace the history of
this type of asymptotics to Pockels (1891) (who proved the
discreteness of the specturm of the Dirichlet Laplacian
\cite{Ben1}), Lord Rayleigh (1905), Sommerfeld (1910), and Lorentz
(1910). Many asymptotics of this type were developed by Courant
and Hilbert \cite{CH}, Pleijel, and Minakshisundaram (see
\cite{BH} for further insight and references).

In 1954, P\'olya conjectured \cite{Patterns} that, for $k=1, 2,
\ldots,$ (see also the series of papers \cite{PolyaConj} \cite{Polya2}
\cite{Polya3})
\begin{equation} \label{pc}
\mu_{k+1} \le \dfrac{4 \pi^2 k^{2/n}}{\left(C_n
|\Om|\right)^{2/n}} \le \la_k.
\end{equation}
He proved his conjecture for the case of tiling domains in a paper
dedicated to Weyl in 1961 \cite{PolyaConj}. The restrictive
conditions for the Neumann case in \cite{PolyaConj} (see also
\cite{Polya3}) were relaxed and the proof was refined and
simplified by Kellner \cite{Kellner}, ``to P\'olya's pleasure and
satisfaction,'' reports Hersch (see p.~523 of \cite{Polya}). In
1984, Urakawa \cite{Ur} refined the Dirichlet bound to
\begin{equation}
\la_k \ge \delta_L(\Om)^{2/n} \ \dfrac{4 \pi^2 k^{2/n}} {\left(C_n
|\Om|\right)^{2/n}}
\end{equation}
where $\delta_L(\Om)$ is the {\em lattice packing density} of
$\Om$ ($\delta_L(\Om)=1$ for a tiling domain).

Also in 1954, Payne conjectured \cite{Payne} that, independently
of the Weyl term in (\ref{pc}),
\begin{equation}
\mu_{k+1} \le \la_k.
\end{equation}
This question was settled by Friedlander \cite{Friedlander} in
1991 for domains with smooth boundaries. More recently, Friedlander's
breakthrough was generalized for domains with non-smooth boundaries by
Filonov \cite{Filo}. On a different track, Li and Yau \cite{LY} proved
that, for $k\ge1$,
\begin{equation}
\sum_{i=1}^k \la_i \ge \dfrac{n}{n+2} \dfrac{4 \pi^2 k^{1+2/n}}
{\left(C_n |\Om|\right)^{2/n}},
\end{equation}
from which it obtains that
\begin{equation}
\la_k \ge \dfrac{n}{n+2} \dfrac{4 \pi^2 k^{2/n}} {\left(C_n
|\Om|\right)^{2/n}}.
\end{equation}

Inequalities (\ref{new1}), (\ref{new2}) and (\ref{new4}) can then
be thought of as counterparts to these two inequalities of Li-Yau.

In 1992, Kr\"oger \cite{Kr} produced the Neumann parallels to the
Li-Yau inequalities. For $k\ge 1$, he proved
\begin{equation}
\sum_{i=1}^k \mu_i \le \dfrac{n}{n+2} \dfrac{4 \pi^2 k^{1+2/n}}
{\left(C_n |\Om|\right)^{2/n}}, \label{krog1}
\end{equation}
and
\begin{equation}
\mu_{k+1} \le \left(1+\dfrac{n}{2}\right)^{2/n} \dfrac{4 \pi^2
k^{2/n}} {\left(C_n |\Om|\right)^{2/n}}. \label{krog2}
\end{equation}
Notice that (\ref{krog2}) implies that
\begin{eqnarray*}
\sum_{j=1}^k \mu_{j} &\le \left(1+\dfrac{n}{2}\right)^{2/n}
\dfrac{4 \pi^2}{\left(C_n |\Om|\right)^{2/n}} \
\left(\sum_{j=0}^{k-1} \
j^{2/n} \right) \\
\notag &< \left(1+\dfrac{n}{2}\right)^{2/n} \ \dfrac{4
\pi^2}{\left(C_n |\Om|\right)^{2/n}} \ \dfrac{k^{1+2/n}}{1+2/n}.
\end{eqnarray*}
Again, we have used the left Riemann sum comparison
(\ref{left-riemann}) in these two inequalities. Kr\"oger's
inequality (\ref{krog1}) is tighter, since
\[\dfrac{\left(1+n/2\right)^{2/n}}{1+2/n} > \dfrac{1}{1+2/n}= \dfrac{n}{n+2}.\]
Thus, the ``averaged'' version (namely (\ref{krog1})) of
Kr\"oger's two inequalities is sharper.

Our bounds (\ref{new1}) and (\ref{new4}) are also related to the
result of Ashbaugh and Benguria \cite{AB4} who proved, for $m\ge
1$,
\begin{equation}
\dfrac{\la_{2^m}}{\la_1} \le
\left(\dfrac{j_{n/2,1}^2}{j_{n/2-1,1}^2}\right)^m. \label{ab94}
\end{equation}
Of course, one cannot expect to fare better in the case of $m=1$
since this is another conjecture by Payne, P\'olya, and Weinberger
\cite{PPW1} \cite{PPW2} (herein referred to as PPW) which was settled
by Ashbaugh and Benguria
\cite{ABPPW} (see also \cite{AB0}) in 1991, namely
\begin{equation}
\dfrac{\la_2}{\la_1} \le \dfrac{j_{n/2,1}^2}{j_{n/2-1,1}^2}.
\label{AB-Ineq}
\end{equation}
The ratio on the RHS of (\ref{AB-Ineq}) is that for the two first
eigenvalues of an $n$-ball. It has the asymptotic expansion
\cite{AB5}
\begin{equation} \label{ppw-expansion}
\dfrac{j_{n/2,1}^2}{j_{n/2-1,1}^2}
\sim 1+ \dfrac{4}{n} \ - \dfrac{4}{3} \ c_1
\dfrac{2^{5/3}}{n^{5/3}}+ \dfrac{12}{n^2}+ \dfrac{4}{3} (c_1^2-2
c_2) \dfrac{2^{7/3}}{n^{7/3}}+ O\left(n^{-8/3}\right),
\end{equation}
where $c_1 \approx 1.8557571$ and $c_2 \approx 1.033150$ (see
\cite{AS}).

Payne, P\'olya, and Weinberger \cite{PPW1} \cite{PPW2} (see also
\cite{AB4} \cite{Thom}) proved the weaker form
\begin{equation}
\la_{k+1}- \la_k \le \dfrac{4}{n k} \ \sum_{j=1}^k \ \la_j,
\end{equation}
from which one can infer that
\[\dfrac{\la_{k+1}}{\la_k} \le 1+ 4/n,\]
and
\begin{equation} \label{ppw}
\dfrac{\la_{k}}{\la_1} \le (1+ 4/n)^{k-1}.
\end{equation}
Note that (\ref{ab94}) can be put in the form
\begin{equation} \label{ab94-comp1}
\dfrac{\la_{k}}{\la_1} \le
\left(\dfrac{j_{n/2,1}^2}{j_{n/2-1,1}^2}
\right)^{\left[\dfrac{\log k}{\log 2}\right]}
\end{equation}
where $\left[ x \right]$ stands for the integer part of $x$. This
bound can be thought of as one of the form
\begin{equation} \label{ab94-comp2}
\left(\dfrac{j_{n/2,1}^2}{j_{n/2-1,1}^2} \right)^{\dfrac{\log
k}{\log 2}}=k^{\dfrac{\log{ \left(j_{n/2,1}^2
/j_{n/2-1,1}^2\right)}}{\log 2}}.
\end{equation}
By virtue of the expansion (\ref{ppw-expansion}), the power
$\dfrac{\log{ \left(j_{n/2,1}^2 /j_{n/2-1,1}^2\right)}}{\log 2}$
has the asymptotic form
\[\dfrac{1}{\log 2}
\left(\dfrac{4}{n} \ - \dfrac{4}{3} \ c_1 \dfrac{2^{5/3}}{n^{5/3}}
\notag + \dfrac{4}{n^2}+ \dfrac{4}{3} (c_1^2-2 c_2)
\dfrac{2^{7/3}}{n^{7/3}}\right)+ O\left(n^{-8/3}\right),\]
 or
\begin{equation} \label{ab-power}
\dfrac{\log{ \left(j_{n/2,1}^2 /j_{n/2-1,1}^2\right)}}{\log 2}
\sim \dfrac{5.77078}{n} \ - 6.10703 \ c_1 \dfrac{1}{n^{5/3}}+
 O\left(\dfrac{1}{n^2}\right).\end{equation}

Thus, while tight at the bottom of the spectrum (viz.
(\ref{AB-Ineq})), (\ref{ab94}) does not capture the expected Weyl
behavior of $k^{2/n}$. Inequalities (\ref{new1}) and (\ref{new4})
remedy this.

The key to the new results is an observation by Ashbaugh and
Benguria--the extent and limitations of which are discussed on
p.~561 of \cite{AB8}. If one identifies $\mu_{k+1}$ with
$\la_{k+1} - \la_1$ and $|\Om|^{-2/n}$ with $\la_1$, then the RHS
of the PPW inequality (\ref{AB-Ineq}) can be seen as maximizing
the ratio $\la_2/\la_1$ in the same vein as the quantity
$C_n^{2/n} p_{n/2,1}^2$ maximizes $|\Om|^{-2/n} \ \mu_2$ for any
domain $\Om$ ($p_{\nu,k}$ denotes the $k$th positive zero of the
derivative of $x^{1-\nu} J_{\nu}(x)$ and $C_n$ is as defined
above, i.e. the volume of the unit $n$-ball). The latter is a
result of Szeg\H{o} in $2$ dimensions and Weinberger in $n$
dimensions. The maximum for both is assumed when $\Om$ is an
$n$-ball. The strategy of proof for both is similar though the
first is considerably more involved \cite{ABPPW} \cite{AB0}. This
loose analogy can also be seen in comparing the methods of proof
and results for
\begin{equation}
\sum_{j=1}^n (\la_{j+1} - \la_1) \le 4 \, \la_1
\label{sum}
\end{equation}
and
\begin{equation}
\sum_{j=1}^n \mu_{j+1} \le n \
\left(\dfrac{C_n}{|\Om|}\right)^{2/n} p_{n/2,1}^2 \label{abmu2}
\end{equation}
both of which were proved by Ashbaugh and Benguria in \cite{AB5}
and \cite{AB8}. Inequality (\ref{sum}) is the extension to $n$
dimensions of a result in \cite{PPW2}. (Note that (\ref{abmu2})
was proved with the further restriction that $\Om$ is invariant
with respect to 90$^o$ rotations in the coordinate planes.) Our
new inequality (\ref{new2}) can be viewed as an extension for
$k\neq n$ of (\ref{sum}). (See Section \ref{comparison} for a
comparison with existing results.)

The loose correspondence can also be traced in the analogy between
\begin{equation}
\sum_{k=1}^n \dfrac{1}{\la_{k+1} - \la_1} \ge \dfrac{2
j_{n/2-1,1}^2+ n (n-4)}{6 \la_1} > \dfrac{n^2}{4 \la_1}
\label{abchiti}
\end{equation}
and
\begin{equation}
\sum_{k=1}^n \dfrac{1}{\mu_{k+1}} \ge \dfrac{ n \
\left(\dfrac{|\Om|}{C_n}\right)^{2/n}}{ p_{n/2,1}^2}.
\label{abmu1}
\end{equation}
Both of these bounds are also results found in \cite{AB5} and
\cite{AB8} (with (\ref{abmu1}) also true under rotational symmetry
of the base domain $\Om$). Inequality (\ref{abchiti}) is an
extension and improvement of earlier results of Chiti \cite{Chi4}.
The $n^2/4 \la_1$ term in (\ref{abchiti}) is what corresponds to
(\ref{sum}) via the ``usual Cauchy-Schwarz connection'' (viz. the
proof of Cor.~\ref{Chiti-Ashbaugh-Benguria}). On the other hand,
there is also a conjectured inequality, from which, if proved,
(\ref{abchiti}) would follow via the Cauchy-Schwarz argument. That
inequality would be (\ref{sum}) but with its RHS replaced by
\[\dfrac{6 n^2 \la_1}{2 j_{n/2-1,1}^2+ n (n-4)}.\]

Our new inequality (\ref{new3}) can be viewed as an extension, for
$k\neq n$, of the Ashbaugh-Benguria-Chiti inequality
(\ref{abchiti}).

We complete this section by giving the asymptotic expansions for
the coefficients appearing in (\ref{new1}), and (\ref{new4}) (see
\cite{AB5} and \cite{Lau} for similar estimates).
\begin{equation} \label{asymp-1}
\left(1+\dfrac{n}{2}\right)^{2/n}~H_n^{2/n}  \sim  1  +
\dfrac{2}{3 n} \log{\dfrac{4 n^4}{b_0^6}} +  2^{8/3} (b_1-c_1)
\left(\dfrac{1}{n}\right)^{5/3}  + O\left(\dfrac{1}{n}\right)^2.
\end{equation}
In the case of (\ref{new4}), the expansion reads
\begin{eqnarray}
\left(1+\dfrac{4}{n}\right) \dfrac{n}{n+2} & H_n^{2/n} \sim 1+
\dfrac{2}{3 n} \ \left(3+ \log{\dfrac{32 n}{b_0^6}}\right) + \\
\notag & 2^{8/3} \ (b_1-c_1) \left(\dfrac{1}{n}\right)^{5/3}  +
O\left(\dfrac{1}{n}\right)^2.
\end{eqnarray}
Here (see \cite{AS}) $b_0\approx 1.1131028$, $b_1 \approx
1.484606$ and $c_1 \approx 1.8557571$.

\section{The Counting Function} \label{counting}

One can motivate these inequalities in terms of the {\em counting
function},
\[N(\la) = \sum_{\la_k \le \la} \ 1 = \sup_{\la_k \le \la} \ k. \]
Our Theorem \ref{theo1} can then be restated.
\begin{theorem} \label{counting2}
For $\la\ge \la_1$, we have the lower bounds
\begin{equation} \label{count1}
N(\la) > \dfrac{2}{n+2} \ \dfrac{1}{H_n} \
\left(\la-\la_1\right)^{n/2} \ \la_1^{-n/2},
\end{equation}
and
\begin{equation} \label{count2}
N(\la) > \left(\dfrac{n+2}{n+4}\right)^{n/2} \ \dfrac{1}{H_n} \
\left(\la-(1+4/n) \ \la_1\right)^{n/2} \ \la_1^{-n/2}.
\end{equation}
\end{theorem}

\noindent {\em Remark.} While (\ref{count1}) is a direct corollary
to an earlier result of Laptev (see Cor.~4.4 in \cite{Lap}) and
Chiti's inequality (\ref{Chiti-ess-sup}) below (see \cite{Chi3},
\cite{Chi4}, \cite{AH1}), (\ref{count2}) is new and in fact
sharper. We refer the reader to the discussion in Section
\ref{comparison}.

\begin{proof}
Inequalities (\ref{count1}) and (\ref{count2}) follow from
(\ref{new1}) and (\ref{new4}), respectively. The proof for both is
similar. We show  (\ref{count1}) for illustration. Let $N(\la)=j$.
Then, from the definition of $N(\la)$, $\la_{j+1}>\la$. Inequality
(\ref{new1}) then implies
\[\la - \la_1 < \left(1+\dfrac{n}{2}\right)^{2/n} H_n^{2/n} \,
\la_1  \, j^{2/n}.
\]
Reversing, one gets (\ref{count1}).
\end{proof}
\noindent {\em Remark.} Using an entirely different method,
Safarov (see ineq. (2.9) of \cite{Safarov}) proved,
\begin{equation}
N(\la) \ge \dfrac{2}{n+2} \ e^{-1/{4 \pi}} \ L_n^{cl}
\left(\la-\la_1\right)^{n/2} \ \la_1^{-n/2},
\label{saf}
\end{equation}
where $L_n^{cl}= C_n/(2 \pi)^n$. We have listed in Table
\ref{tabla} numerical values of the coefficients appearing in
(\ref{count2}), (\ref{count1}), and (\ref{saf}).

\begin{table} [htbp]
  \begin{center}
   \begin{tabular}{|c|c|c||c|}
    \hline \hline
$n$ &$\left(\frac{n+2}{n+4}\right)^{n/2} \frac{1}{H_n}$ & $\frac{2}{n+2} \frac{1}{H_n}$  & $\frac{2}{n+2} e^{-1/{4 \pi}} \ L_n^{cl}$ \\
    \hline
2 & 0.259775 & 0.194831   & 0.036745    \\
3 & 0.201227 & 0.133333   & 0.062381  \\
4 & 0.167459 & 0.099235   & 0.000975      \\
5 & 0.145412 & 0.077874   & 0.000142       \\
6 & 0.129833 & 0.063395   & 0.000019     \\
7 & 0.118201 & 0.053193   & $2.5 \times 10^{-6}$     \\
    \hline
\end{tabular}
\end{center}
\caption{ Comparison of the coefficients appearing in (\ref{count2}), (\ref{count1}), and (\ref{saf}).}
    \label{tabla}
\end{table}

\noindent These inequalities complete bounds of the form
\begin{equation}
N(\la) \le \tilde{K}_n \ \la^{n/2} \ |\Om|.
\end{equation}
found in the works of Lieb \cite{Lieb} and Li-Yau \cite{LY} (see
Laptev \cite{Lap}). Of course, Weyl's asymptotic formula reads
\begin{equation}
N(\la) \sim \dfrac{C_n |\Om| \la^{n/2}}{(2 \pi)^n} = L^{cl}_n \
|\Om| \la^{n/2}.
\end{equation}
While the P\'olya conjecture states
\begin{equation}
N(\la) \le L^{cl}_n \la^{n/2}  \ |\Om|.
\end{equation}
The Li-Yau bounds can be reformulated as (see \cite{Lap})
\begin{equation}
N(\la) \le \left(\dfrac{n+2}{n}\right)^{n/2} \ L^{cl}_n \
\la^{n/2} \ |\Om|. \label{LY}
\end{equation}
In the same spirit, one should note F. Berezin's inequality
\cite{Berezin} (see \cite{Safar})
\begin{equation}
\displaystyle{\int_0^{\la} N(\mu) \, d \mu} \le \frac{2}{n+2}
L^{cl}_n \ \la^{n/2+1} \ |\Om|. \label{Ber}
\end{equation}
Laptev \cite{Lap} and Safarov \cite{Safar} have noted that the
Li-Yau bound (\ref{LY}) is a corollary of (\ref{Ber}). Indeed, for
$\theta>0$,
\begin{equation}
N(\la) \le \frac{1}{\theta \la} \displaystyle{\int_0^{\la + \theta
\la} N(\mu) d \mu} \le \dfrac{2 \left(1+
\theta\right)^{n/2+1}}{\left(n+2\right) \theta} L^{cl}_n |\Om|
\la^{n/2}.
\end{equation}
Li-Yau's bound follows by setting $\theta=2/n$. As for the
Ashbaugh-Benguria inequality (\ref{ab94}), it can be reworked to
appear in the following terms (see \cite{AB4}): For $\la \ge
\la_1$,
\begin{equation} \label{ab-count}
N(\la) \ge 2^{\left[
\log(\la/\la_1)/\log(j_{n/2,1}^2/j_{n/2-1,1}^2) \right]}.
\end{equation}
Notice that the RHS of this inequality assumes the form
\[\left(\dfrac{\la}{\la_1}\right)^{\dfrac{1}
{\log_2{j_{n/2,1}^2/j_{n/2-1,1}^2}}}.\] This allows one to restate
(\ref{ab-count}) (in view of (\ref{ab-power})) as
\[N(\la) \ge \left(\dfrac{\la}{\la_1}\right)^{n/5.77078}.\]
In fact, for $\la\ge \la_2$, Ashbaugh and Benguria have the sharper
bound \cite{AB4}
\begin{equation} \label{ab-count-2}
N(\la) \ge 2^{1+ \left[
\log(\la/\la_2)/\log(j_{n/2,1}^2/j_{n/2-1,1}^2) \right]}.
\end{equation}
which, in view of the above considerations, reads as
\[N(\la) \ge 2 \ \left(\dfrac{\la}{\la_2}\right)^{n/5.77078}.\]

\section{Proof of Theorem \ref{theo1}} \label{proof-of-main}

We begin with the Rayleigh-Ritz estimate for $\la_{k+1}$,
\begin{equation} \label{RR}
\la_{k+1} \le \inf_{r\ge r_0} \dfrac{\int_{B_r} \int_{\Om} \ | \Q
\phi |^2 d\xx \ d\zz}{\int_{B_r} \int_{\Om} |\phi|^2 d\xx \ d\zz},
\end{equation}
where $B_r=$ is a ball of radius $r\ge r_0$, and $r_0 = H_n^{1/n}
\ (1+k)^{1/n} \ \sqrt{\la_1}$. This characterization is suggested
by considerations similar to \cite{Kr}. In fact, the bulk of the
arguments follow steps described there. The {\em test function}
$\phi$ is required to satisfy
\[\phi \perp u_1, u_2, \cdots, u_k.\]
It is chosen to be of the form
\[\phi= e^{i \xx \cdot \zz }
u_1(\xx) - \sum_{j=1}^k a_j(\zz) \ u_j(\xx).\]
The orthogonality conditions lead to $a_j(\zz)= \int_{\Om} \ u_1
\overline{u_j} e^{i \xx \cdot \zz } d\xx$. We calculate
\begin{eqnarray} \label{bottom}
\int_{\Om} |\phi|^2 &=& \int_{\Om} |u_1|^2 - 2 \sum_{j} |a_j|^2 +
\sum_{j,\ell} a_j \overline{a_{\ell}} \ \int_{\Om} u_j
\overline{u_{\ell}}
\\ \notag &=& 1 -\sum_{j=1}^k |a_j|^2,
\end{eqnarray}
since $\int_{\Om} u_j \overline{u_{\ell}}=\delta_{j \ell}$. One
has
\[\int_{\Om} |\Q
\phi |^2  = \dfrac{1}{2} \ \left(\int_{\Om} \phi (-\D
\overline{\phi})+ \int_{\Om} (-\D \phi) \overline{\phi}\right),\]
since $\phi$ and $\overline{\phi}$ satisfy the Dirichlet boundary
condition. We let $\phi_0=e^{i \xx \cdot \zz }$. Then,
\begin{eqnarray*}
-\D \phi &=& (-\D \phi_0) u_1- 2 \Q \phi_0 \cdot \Q u_1 + \la_1
\phi_0 u_1 - \sum a_j \la_j u_j \\ \notag &=& (-\D \phi_0) u_1- 2
\Q \phi_0 \cdot \Q u_1 + \la_1 \left(\phi+\sum a_j u_j\right) -
\sum a_j \la_j u_j \\ \notag &=& (-\D \phi_0) u_1- 2 \Q \phi_0
\cdot \Q u_1 + \la_1 \phi - \sum a_j (\la_j - \la_1) u_j.
\end{eqnarray*}
Similarly,
\begin{equation}
-\D \overline{\phi} =(-\D \overline{\phi}_0) \overline{u}_1- 2 \Q
\overline{\phi}_0 \cdot \Q \overline{u}_1 + \la_1 \overline{\phi}
- \sum \overline{a}_j (\la_j - \la_1) \overline{u}_j. \notag
\end{equation}
Therefore,
\begin{equation}
\int_{\Om} (-\D \phi) \overline{\phi} = \la_1 \ \int_{\Om}
|\phi|^2+ \int_{\Om}  \left(-\D \phi_0 u_1- 2 \Q \phi_0 \cdot \Q
u_1\right) \ \overline{\phi} -\sum a_j (\la_j - \la_1) \int_{\Om}
u_j \overline{\phi}. \notag
\end{equation}
Orthogonality makes $\int_{\Om} u_j \overline{\phi}= 0$.
\begin{equation} \label{half}
\int_{\Om} (-\D \phi) \overline{\phi} = \la_1 \ \int_{\Om}
|\phi|^2+ \int_{\Om}  \left(-\D \phi_0 u_1- 2 \Q \phi_0 \cdot \Q
u_1\right) \ \overline{\phi}.
\end{equation}
We now concentrate on the quantity $\int_{\Om}  \left(-\D \phi_0
u_1- 2 \Q \phi_0 \cdot \Q u_1\right) \ \overline{\phi}$. It is
equal to
\[
\int_{\Om}  \left(-\D \phi_0 \ \overline{\phi}_0 |u_1|^2- 2 \Q
\phi_0 \cdot \Q u_1 \ \overline{\phi}_0 \ \overline{u}_1 \right) -
\sum \overline{a}_j \ \int_{\Om} \left(-\D \phi_0 u_1- 2 \Q \phi_0
\cdot \Q u_1\right) \ \overline{u}_j.
\]
Since \[\dfrac{\partial \phi_0}{\partial x_{\ell}} = i \ z_{\ell}
\ \phi_0,\] and
\[-\D \phi_0 = |\zz|^2 \phi_0,\]
it obtains that
\begin{eqnarray} \label{second-half}
\int_{\Om}  \left(-\D \phi_0 u_1 - 2 \Q \phi_0 \cdot \Q u_1\right)
\ \overline{\phi}  &= |\zz|^2 \ \int_{\Om} |u_1|^2 + 2 \ i \
\sum_{\ell=1}^n z_{\ell} \ \int_{\Om} \overline{u}_1 \
\dfrac{\partial u_1}{\partial x_{\ell}} \notag \\
&- \sum_{j=1}^k \overline{a}_j \int_{\Om} \left(-\D \phi_0 u_1- 2
\Q \phi_0 \cdot \Q u_1\right) \ \overline{u}_j.
\end{eqnarray}
We now use the identity $-\D \phi_0 u_1- 2 \Q \phi_0 \cdot \Q u_1
= - \D (\phi_0 u_1) + \phi_0 \D u_1$ (and the Dirichlet boundary
condition) to reduce the second integral to
\[\int_{\Om} \left(-\D \phi_0 u_1- 2 \Q \phi_0 \cdot
\Q u_1\right) \ \overline{u}_j = \left(\la_j - \la_1\right) \
a_j.\] Substituting this into (\ref{second-half}) and ultimately
in (\ref{half}) we conclude (since $\int_{\Om} |u_1|^2=1$)
\begin{equation} \label{half2}
\int_{\Om} (-\D \phi) \overline{\phi} = \la_1 \ \int_{\Om}
|\phi|^2+ |\zz|^2 + 2 \ i \ \sum_{\ell=1}^n z_{\ell} \ \int_{\Om}
\overline{u}_1 \ \dfrac{\partial u_1}{\partial x_{\ell}} -
\sum_{j=1}^k |a_j|^2 \left(\la_j - \la_1\right).
\end{equation}
The term $\int_{\Om} \overline{u}_1 \ \dfrac{\partial
u_1}{\partial x_{\ell}}$ is of course real, since $u_1$ was
assumed to be real. Conjugating, we obtain
\begin{equation} \label{half3}
\int_{\Om} \phi (-\D \overline{\phi}) = \la_1 \ \int_{\Om}
|\phi|^2+ |\zz|^2 - 2 \ i \ \sum_{\ell=1}^n z_{\ell} \ \int_{\Om}
\overline{u}_1 \ \dfrac{\partial u_1}{\partial x_{\ell}} -
\sum_{j=1}^k |a_j|^2 \left(\la_j - \la_1\right).
\end{equation}
Substituting (\ref{bottom}), (\ref{half2}), and (\ref{half3}) into
the Rayleigh-Ritz ratio (\ref{RR}) yields
\begin{equation}
\la_{k+1} -\la_1 \le \dfrac{\int_{B_r} \left( |\zz|^2 -
\sum_{j=1}^k |a_j(\zz)|^2 \ (\la_j - \la_1) \right)
d\zz}{\int_{B_r} \left(1-\sum_{j=1}^k |a_j(\zz)|^2 \right) \
d\zz}, \notag
\end{equation}
or
\begin{equation} \label{first-reduction}
\la_{k+1} -\la_1  \le \dfrac{\int_{B_r} |\zz|^2 d\zz -
\sum_{j=1}^k \left(\la_j - \la_1\right) \ \int_{B_r} |a_j(\zz)|^2
d\zz } {|B_r|-\sum_{j=1}^k \int_{B_r} |a_j(\zz)|^2 \ d\zz}.
\end{equation}

We are now ready for our second reduction. We rescale the Fourier
coefficient $a_j(\zz)$ by defining
\[
\tilde{a}_j(\zz) = \dfrac{1}{(2 \pi)^{n/2}} \ \int_{\Om} \ u_1
\overline{u_j} e^{i \xx \cdot \zz } d\xx = \dfrac{a_j(\zz)}{(2
\pi)^{n/2}}.\]
By Parseval's identity
\begin{eqnarray*}
\dfrac{1}{(2 \pi)^{n}} \ \int_{B_r} |a_j(\zz)|^2  d\zz  & = & \int_{B_r} |\tilde{a}_j(\zz)|^2  d\zz \\
& \le  & \int_{\rz^n} |\tilde{a}_j(\zz)|^2  d\zz \\ \notag & = &
\int_{\Om} |u_1|^2 |u_j|^2 d\xx \\ \notag & \le & \Ess |u_1|^2
\int_{\Om}  |u_j(\xx)|^2 d\xx.
\end{eqnarray*}

We now use the following result of Chiti \cite{Chi3} (see also
\cite{Chi4}, \cite{AH1}),

\begin{equation}
\Ess |u_1| \le \left(\dfrac{\la_1}{\pi}\right)^{n/4}
\dfrac{2^{1-n/2}}{\Gamma(n/2)^{1/2} j_{n/2-1,1} J_{n/2}
(j_{n/2-1,1})}. \label{Chiti-ess-sup}
\end{equation}
Therefore,
\begin{equation} \label{second-reduction}
0<\int_{B_r} |a_j(\zz)|^2 d\zz \le \ctilde \ \la_1^{n/2},
\end{equation}
where
\[\ctilde=\frac{2^{2} \ \pi^{n/2}}{\Gamma(n/2) \ j_{n/2-1,1}^2 \
J_{n/2}^2 (j_{n/2-1,1})}.\] We note that \[\pi^{n/2}=\dfrac{n C_n
\Gamma(n/2)}{2}\] ($C_n$ is the volume of the unit $n-$ball).
Moreover, the constant $H_n$ defined in (\ref{constant}) is given
by
\[H_n= \dfrac{\ctilde}{C_n}.\]
\noindent {\em Remark.} Safarov obtained (\ref{saf}) using the
following result of E.~B.~Davies \cite{Davies}
\begin{equation}
\Ess |u_1| \le e^{1/{8 \pi}}\, \la_1^{n/4}.
\label{davies}
\end{equation}
Chiti's statement (\ref{Chiti-ess-sup}) is an isoperimetric
inequality. It saturates when $\Om$ is an $n-$ball. Note that
$e^{1/{8 \pi}} \approx 1.04059$, while the constant in
(\ref{Chiti-ess-sup}) takes the values listed in Table
\ref{table-chiti}.

\begin{table} [htbp!]
  \begin{center}
   \begin{tabular}{|c|c|}
    \hline \hline
$n$ & Chiti bound \\
    \hline
2 &  0.451909    \\
3 &  0.225079 \\
4 &  0.103129    \\
5 &  0.044409     \\
6 &  0.018199    \\
7 &  0.007157    \\
    \hline
\end{tabular}
\end{center}
\caption{Values of the constant in Chiti's bound
(\ref{Chiti-ess-sup}) as a function of the dimension $n$.}
    \label{table-chiti}
\end{table}

We now prove by induction the following lemma.
\begin{lemma} \label{main-lemma}
For $r\ge r_0(k)= H_n^{1/n} \ (1+k)^{1/n} \ \sqrt{\la_1}$, $k\ge
1$,
\begin{equation} \label{fourth-reduction}
\la_{k+1} -\la_1  \le \dfrac{\frac{n}{n+2} \ C_n \ r^{n+2} -
\ctilde \ \la_1^{n/2} \ \sum_{j=1}^k \left(\la_j - \la_1\right)}
{C_n r^n- k \ \ctilde \ \la_1^{n/2}}.
\end{equation}
\end{lemma}

\begin{proof}
We first note that \[\int_{B_r} |\zz|^2 \ d\zz=
\dfrac{n}{n+2} \ C_n \ r^{n+2},\] while \[|B_r|=\int_{B_r} \ d\zz=
C_n \ r^n.\] These two facts reduce (\ref{first-reduction}) to
\begin{equation} \label{third-reduction}
\la_{k+1} -\la_1  \le \dfrac{\dfrac{n}{n+2} \ C_n \ r^{n+2}-
\sum_{j=1}^k \left(\la_j - \la_1\right) \ \int_{B_r} |a_j(\zz)|^2
d\zz } {C_n \ r^n-\sum_{j=1}^k \int_{B_r} |a_j(\zz)|^2 \ d\zz}.
\end{equation}
For $k=1$
\begin{equation}
\la_{2} -\la_1  \le \dfrac{\dfrac{n}{n+2} \ C_n \ r^{n+2}}{C_n \
r^n- \int_{B_r} |a_1(\zz)|^2 \ d\zz}. \notag
\end{equation}
It then obtains by virtue of (\ref{second-reduction}) that, for $r\ge
r_0(1)=H_n^{1/n} \ 2^{1/n} \sqrt \la_1$ (this condition guarantees the
denominator is positive)
\begin{equation} \label{induction-1}
\la_{2} -\la_1  \le \dfrac{\dfrac{n}{n+2} \ C_n \ r^{n+2}} {C_n \
r^n-\ctilde \ \la_1^{n/2}},
\end{equation}
as desired. Suppose now that, for $r\ge r_0(k-1)= H_n^{1/n} \ k^{1/n} \
\sqrt{\la_1}$,
\begin{equation}
\la_{k} -\la_1  \le \dfrac{\frac{n}{n+2} \ C_n \ r^{n+2} -
\ctilde \ \la_1^{n/2} \ \sum_{j=1}^{k-1} \left(\la_j - \la_1\right)}
{C_n r^n- (k-1) \ \ctilde \ \la_1^{n/2}}. \notag
\end{equation}
Then, this is also true for $r\ge r_0(k)$ as well (since
$r_0(k) >r_0(k-1)$). This implies
\begin{equation}
\la_{k} -\la_1  \le \dfrac{\frac{n}{n+2} \ C_n \ r^{n+2} -
\ctilde \ \la_1^{n/2} \ \left(\sum_{j=1}^{k-1} \left(\la_j -
\la_1\right)+ \left(\la_k- \la_1 \right) \right)}
{C_n r^n- (k-1) \ \ctilde \ \la_1^{n/2} - \ctilde \ \la_1^{n/2}}, \notag
\end{equation}
or,
\begin{equation} \label{induction-2}
\la_{k} -\la_1  \le \dfrac{\frac{n}{n+2} \ C_n \ r^{n+2} -
\ctilde \ \la_1^{n/2} \ \sum_{j=1}^{k} \left(\la_j -
\la_1\right)}{C_n r^n- k \ \ctilde \ \la_1^{n/2}}.
\end{equation}
(We have used the equivalence
$\alpha \le \frac{A}{B} \Leftrightarrow \alpha \le \frac{A- \alpha \beta}{B- \beta}$, for $B- \beta \ge 0$.)

\noindent We also notice that if $A_j\ge0$ then
\[\dfrac{\sum_{j=1}^k A_j \left(\la_j - \la_1\right)}{\sum_{j=1}^k A_j}
\le \la_k - \la_1.\]
Hence, by virtue of (\ref{induction-2}), and for $A_j=\ctilde \
\la_1^{n/2} - \int_{\Om} |a_j(\zz)|^2 d \zz$,
\begin{equation}
\dfrac{ \sum_{j=1}^k \left(\la_j - \la_1\right) \ \left(\ctilde \
\la_1^{n/2} - \int_{\Om} |a_j(\zz)|^2 d \zz \right)}{\sum_{j=1}^k
\left(\ctilde \ \la_1^{n/2} - \int_{\Om} |a_j(\zz)|^2 d \zz
\right)} \le \dfrac{\frac{n}{n+2} \ C_n \ r^{n+2} - \ctilde \
\la_1^{n/2} \ \sum_{j=1}^{k} \left(\la_j - \la_1\right)}{C_n r^n-
k \ \ctilde \ \la_1^{n/2}}. \label{ind}
\end{equation}
For simplicity we let \[A=\frac{n}{n+2} \ C_n \ r^{n+2} - \ctilde
\ \la_1^{n/2} \ \sum_{j=1}^{k} \left(\la_j - \la_1\right),\]
\[B=C_n r^n- k \ \ctilde \ \la_1^{n/2},\]
\[C= \sum_{j=1}^{k} \left(\la_j-\la_1\right) \ \left( \ctilde \
\la_1^{n/2} - \int_{\Om} |a_j(\zz)|^2 d \zz \right), \] and
\[D=\sum_{j=1}^k \left( \ctilde \ \la_1^{n/2} -
\int_{\Om} |a_j(\zz)|^2 d \zz \right).\]
 By (\ref{third-reduction}),
\begin{equation}
\la_{k+1}- \la_1 \le \dfrac{A+C}{B+D}. \label{ind2}
\end{equation}
Combining (\ref{ind}) and (\ref{ind2}), (since $A/B = \Lambda$ and
$C/D \le \Lambda$ imply $(A+C)/(B+D)\le \Lambda$), one obtains
(\ref{fourth-reduction}) and the proof of the lemma is now
complete.
\end{proof}

If we set $r=r_0(k)=\left(\dfrac{\ctilde}{C_n}\right)^{1/n} \
(1+k)^{1/n} \ \sqrt{\la_1}$ in the statement of Lemma
\ref{main-lemma}, we obtain (\ref{new2}) (the ``averaged''
version), in the form,
\begin{equation}
\sum_{j=1}^{k+1} \left(\la_{j}-\la_1\right) \le \dfrac{n}{n+2}
\left(\dfrac{\ctilde}{C_n}\right)^{2/n} \la_1 (1+k)^{1+2/n}.
\label{discuss-1}
\end{equation}
This choice amounts to setting
\[C_n r^n- k \ \ctilde \ \la_1^{n/2}=\ctilde \ \la_1^{n/2}.\]
If we drop the sum in (\ref{fourth-reduction}) and let
$r=\tilde{r}(k)=\left(\dfrac{\ctilde}{C_n}\right)^{1/n} \ k^{1/n}
\ \sqrt{\la_1} \ \left(1+\dfrac{n}{2}\right)^{1/n}$ we are led to
(\ref{new1}), namely,
\begin{equation}
\la_{k+1}-\la_1 \le \left(1+\dfrac{n}{2}\right)^{2/n}
\left(\dfrac{\ctilde}{C_n}\right)^{2/n} \la_1 k^{2/n}.
\label{discuss-2}
\end{equation}
This choice amounts to making
\[C_n r^n- k \ \ctilde \ \la_1^{n/2}= \dfrac{k \ n}{2} \ \ctilde \
\la_1^{n/2}.\] (Note that $\tilde{r}(k)\ge r_0(k)$ since $n \ge
2$.)

\noindent {\em Remark.} The case $k=1$ in Lemma \ref{main-lemma}
provides a class of bounds for $\la_2- \la_1$ for $r\ge r_0(1)$.
The function
\[\dfrac{\dfrac{n}{n+2} \ C_n \ r^{n+2}} {C_n \ r^n-\ctilde \
\la_1^{n/2}}\] is nonincreasing for $2^{1/n} \ H_n^{1/n} \
\sqrt{\la_1} \le r \le \left(1+n/2\right)^{1/n} \ H_n^{1/n} \
\sqrt{\la_1}$ and nondecreasing beyond. At $r_0(1)$, it assumes
the value of
\[2^{1+2/n} \dfrac{n}{n+2} H_n^{2/n} \ \la_1.\]
Hence
\begin{equation} \label{not-ppw}
\dfrac{\la_2}{\la_1} \le 1+2^{1+2/n} \dfrac{n}{n+2} H_n^{2/n}.
\end{equation}
At its minimum (viz. $r=\left(1+n/2\right)^{1/n} \ H_n^{1/n}  \
\sqrt{\la_1}$), it assumes the form
\begin{equation}
\la_{2}-\la_1 \le \left(1+\dfrac{n}{2}\right)^{2/n} H_n^{2/n} \,
\la_1 \label{not-ppw2}
\end{equation}
Note that (\ref{not-ppw2}) is just (\ref{new1}) for $k=1$. Both
bounds (\ref{not-ppw}) and (\ref{not-ppw2}) are not expected to
fare better than the Ashbaugh-Benguria inequality
(\ref{AB-Ineq})--the best constant of its type (see Table
\ref{table0}). In fact, the first has the asymptotic expansion
\[1+2^{1+2/n} \dfrac{n}{n+2} H_n^{2/n} \sim 3+ \dfrac{2.53636}{n} -
\dfrac{4}{3} \left(\dfrac{1}{n} \ln{\dfrac{1}{n}}\right) -
\dfrac{4.71333}{n^{5/3}} + O\left(\dfrac{1}{n^2}\right).\]
Expanding the second, it obtains (see (\ref{asymp-1}) above)
\begin{equation}
1+ \left(1+\dfrac{n}{2}\right)^{2/n}~H_n^{2/n}  \sim  2  +
\dfrac{0.495591}{n} - \dfrac{8}{3n} \ln{\dfrac{1}{n}} -
\dfrac{2.35666}{n^{5/3}}  + O\left(\dfrac{1}{n^2}\right). \notag
\end{equation}

\begin{table}
  \begin{center}
\begin{tabular}{|c|c|c|c|r|}
    \hline \hline
$n$ & (\ref{ppw})  &  (\ref{new2}) & (\ref{new1}) & (\ref{AB-Ineq}) \\
    \hline
2 &  3     & 6.133  & 6.133 & 2.539    \\
3 &  2.333 & 4.962  & 4.832 & 2.046   \\
4 &  2     & 4.556  & 4.174 & 1.796    \\
5 &  1.8   & 4.171  & 3.777 & 1.645      \\
6 &  1.667 & 3.986  & 3.508 & 1.543  \\
7 &  1.571 & 3.856  & 3.314 & 1.470  \\
    \hline
\end{tabular}
\caption{Bound for $\dfrac{\la_{2}}{\la_1}$ as a function of the
dimension $n$.}
    \label{table0}
\end{center}
\end{table}

\section{Comparison with Existing Results} \label{comparison}

Consider the convex function
\[\nonumber
\phi_{\la}(t)=(\la-t)_{+}=\left\{
\begin{array}{c}
\displaystyle{\la-t,\;\;{\mbox{if}} \;\; t \le \la,}\\
\displaystyle{0 \ ,\;\;{\mbox{if}} \;\;\;t\ge \la.}
\end{array}
\right.
\]
In \cite{Lap}, Laptev proved (see Theo.~4.1) that
\begin{equation} \label{laptev-1}
\sum_{j} \left(\la- \la_j \right)_{+} \ge
\left(\la-\la_1\right)^{1+n/2} L^{cl}_n \frac{2}{n+2}
{\tilde{u}_{1}}^{-2}
\end{equation}
where $\tilde{u}_{1}=\Ess |u_1|$. When $\la=\la_2$,
(\ref{laptev-1}) reduces to (see Cor.~4.2 of \cite{Lap})
\begin{equation} \label{laptev-2}
\la_2-\la_1 \le \left(L^{cl}_n \frac{2}{n+2}\right)^{-2/n}
{\tilde{u}_{1}}^{4/n}.
\end{equation}
Combining this with the isoperimetric inequality of Chiti
(\ref{Chiti-ess-sup}) gives (\ref{not-ppw2}). Note that Chiti's inequality
(\ref{Chiti-ess-sup}) can be put in the form
\begin{equation}
{\tilde{u}_{1}}^{2} \le H_n L^{cl}_n {\la_1}^{n/2}.\label{chiti-myway}
\end{equation}
For $\la\ge \la_1$,
Laptev's result (\ref{laptev-1}) can also be interpreted as (see
Cor.~4.4 of \cite{Lap})
\begin{equation}
N(\la) \ge \left(\la-\la_1\right)^{1+n/2} L^{cl}_n \frac{2}{n+2}
{\tilde{u}_{1}}^{-2}
\end{equation}
which, when combined with Chiti's ineq.~(\ref{chiti-myway}), results in
the statement of (\ref{count1}) with the same universal constant
$\frac{2}{n+2} \ \frac{1}{H_n}$.

To obtain (\ref{new1}) from Laptev's vantage point, set
$\la=\la_{k+1}$ in (\ref{laptev-1}) and observe that
\[k \left(\la_{k+1} - \la_1\right) \ge \sum_{j=1}^{k} \left(\la- \la_j
\right).\] Hence
\begin{equation} \label{laptev-3}
\la_{k+1}-\la_1 \le k^{2/n} \left(L^{cl}_n
\frac{2}{n+2}\right)^{-2/n} {\tilde{u}_{1}}^{4/n}.
\end{equation}
Again, bounding $\tilde{u}_{1}$ using Chiti's isoperimetric
inequality (\ref{chiti-myway}) yields the statement (\ref{new1}). In fact,
one can write (\ref{laptev-1}) in the form
\begin{equation}
k \left(\la_{k+1} - \overline{\la}\right) \ge
\left(\la_{k+1}-\la_1\right)^{1+n/2} L^{cl}_n \frac{2}{n+2}
{\tilde{u}_{1}}^{-2} \notag
\end{equation}
where $\overline{\la}=\sum_{j=1}^k \la_j/k$. Therefore, using
(\ref{Chiti-ess-sup}),
\begin{equation} \label{laptev-4}
\left(1+\dfrac{n}{2}\right) H_n \, \la_1^{n/2}  k \left(\la_{k+1}
- \overline{\la}\right) \ge \left(\la_{k+1}-\la_1\right)^{1+n/2}.
\end{equation}
This is a class of less accessible Weyl-type universal upper
bounds for $\la_{k+1}$ different from both (\ref{new1}) and
(\ref{new2}), but in the same spirit. Indeed, Cor.~4.4 from
\cite{Lap} (and eventually the weaker inequality (\ref{new1}))
follows from Theo.~4.1 of \cite{Lap} by applying the rather rough
estimate $\left(\la -\la_j\right)_{+} \le \left(\la
-\la_1\right)_{+}$. Refining this coarse estimate, one can recover
(\ref{new2}) from Laptev's bound (\ref{laptev-1}). Starting with
(\ref{laptev-1}), one first introduces the Legendre transform
$\mathcal{L}\{f\}(p) = \sup_{\la\ge 0} \left(p \ \la -
f(\la)\right)$. It is then clear (see, e.g., \cite{LapWeidl}) that
\begin{equation}
\mathcal{L}\left\{\displaystyle{\sum_{j}}
\left(\la-\la_j\right)_{+} \right\}(p) = \left(p- [p]\right) \
\la_{[p]+1} + \sum_{j=1}^{[p]} \la_j,\label{legendre1}
\end{equation}
where $[p]$ designates the integer part of $p$. The Legendre
transform of the right hand side of (\ref{laptev-1}) is given by
\begin{equation}
\mathcal{L}\left\{\left(\la-\la_1\right)^{1+n/2} L_n^{cl}
\frac{2}{n+2} \, {\tilde{u}_{1}}^{-2} \right\}(p)= \la_1 \ p +
\frac{n}{n+2} \ p^{1+2/n} \ {\tilde{u}_{1}}^{-4/n} \left(L_n^{cl}
\right)^{-2/n}.
\end{equation}
Since $f(\la) \ge g(\la)$ for all $\la\ge 0$ implies
$\mathcal{L}\left\{f\right\}(p) \le
\mathcal{L}\left\{f\right\}(p)$ for all $p\ge 0$, we have (setting
$p=k$)
\begin{equation}
\sum_{j=1}^k \left(\la_j - \la_1\right) \le \dfrac{n}{n+2} \
{\tilde{u}_{1}}^{4/n} \left(L_n^{cl} \right)^{-2/n} k^{1+2/n}.
\end{equation}
Combining the latter inequality with Chiti's inequality
(\ref{Chiti-ess-sup}) we obtain (\ref{new2}).

\medskip

\begin{figure}[htb!]
  \begin{center}
    \includegraphics[width=.7 \textwidth,totalheight=.35\textwidth, bb= 91 2 322 115]{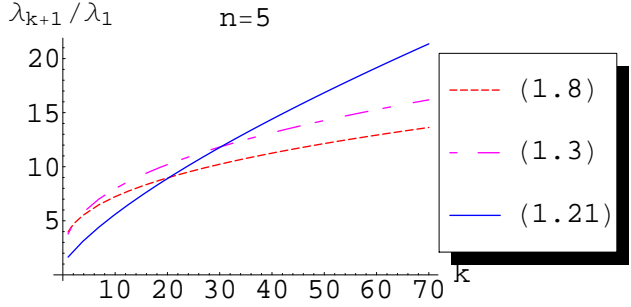}
    \caption{Comparison of New and Old Bounds.}
    \label{fig2}
  \end{center}
\end{figure}

\medskip

Now, we turn to comparing these bounds. We claim that (\ref{new4})
is sharper than (\ref{new1}). To see this, we take the limit of
the ratio of bounds as $k\to \infty$. This limit is equal to
\[\dfrac{1+\frac{4}{n}}{\left(1+\frac{n}{2}\right)^{1+2/n}}.\]
It is strictly less than $1$ since $1+4/n< 1+n/2
<(1+n/2)^{1+2/n}$, for $n\ge 3$. This limit is equal to $3/4$ at
$n=2$.

That both (\ref{new1}) and (\ref{new4}) are sharper than
(\ref{ab94-comp1}) (in the form (\ref{ab94-comp2})) follows from
Krahn's second inequality \cite{Krahn} (see in particular
ineq.~(22) of \cite{A3})
\[|\Om|^{2/n} \la_2 > 2^{2/n} C_n^{2/n} j_{n/2-1,1}^2.\]
For the unit ball in $\rz^n$, $|\Om|=C_n$ and $\la_2=j_{n/2,1}^2$.
Therefore,
\[\dfrac{j_{n/2,1}^2}{j_{n/2-1,1}^2}>2^{2/n},\]
or
\begin{equation} \label{ab-comp-powers}
\dfrac{\log{ \left(j_{n/2,1}^2 /j_{n/2-1,1}^2\right)}}{\log
2}>\dfrac{2}{n}
\end{equation} (see also (\ref{ab-power}) above). This is clearly
displayed in Fig.~2 where the Ashbaugh-Benguria bound fares better
to about $k=20$. The ``averaged'' bound (\ref{new4}) takes over
and--at a latter stage--so does (\ref{new1}). Three tables are
included in this paper which display this fact as well (see Tables
\ref{table0}-\ref{table2}). The new inequalities--both of which
disguised in earlier work of Laptev--cannot be expected to improve
on existing bounds in the case of $\la_2/\la_1$. There is a
competition (see Table \ref{table1}) in the case of
$\la_{32}/\la_1$ between (\ref{new1}) and (\ref{ab94}) (already
(\ref{new4}) is better than both for $n \ge 3$). In the case of
$\la_{128}/\la_1$, both new bounds show considerable improvement
(see Table \ref{table2}).

\begin{table}
  \begin{center}
\begin{tabular}{|c|c|c|c|r|}
    \hline \hline
$n$ &   (\ref{ppw})   &  (\ref{new4}) & (\ref{new1}) & (\ref{ab94}) \\
    \hline
2 &  $6.177 \times 10^{14}$   & 122.334   & 160.112  & 105.46   \\
3 &  $2.554 \times 10^{11}$   & 31.071   & 38.811  & 35.831  \\
4 &  $2.147 \times 10^{9}$   &  15.606  & 18.675  & 18.707   \\
5 &  $8.193 \times 10^7$   &  10.341  & 11.965  & 12.052   \\
6 &  $7.539 \times 10^6$   &  7.870  & 8.878  & 8.758   \\
7 &  $1.217 \times 10^6$   &  6.491  & 7.174  & 6.865  \\
    \hline
\end{tabular}
\caption{ Bound for $\dfrac{\la_{32}}{\la_1}$ as a function of the
dimension $n$.}
    \label{table1}
\end{center}
\end{table}

\begin{table}
  \begin{center}
   \begin{tabular}{|c|c|c|c|r|}
    \hline \hline
$n$ &   (\ref{ppw})  &  (\ref{new4}) & (\ref{new1}) & (\ref{ab94})  \\
    \hline
2 &  $3.930 \times 10^{60}$   & 491.885   & 652.846  & 679.705 \\
3 &  $5.408 \times 10^{46}$   & 75.911  & 97.808  &  149.957 \\
4 &  $1.701 \times 10^{38}$   &  29.539  & 36.774    & 60.369  \\
5 &  $2.628 \times 10^{32}$   &  16.814  & 20.2736    &  32.621  \\
6 &  $1.496 \times 10^{28}$   &  11.593 & 11.5934    & 20.861  \\
7 &  $8.500 \times 10^{24}$   &  8.917 & 8.917    & 14.836  \\
    \hline
\end{tabular}
\caption{ Bound for $\dfrac{\la_{128}}{\la_1}$ as a function of
the dimension $n$.}
    \label{table2}
\end{center}
\end{table}

\begin{center}
\large{Acknowledgment}
\end{center}
The author offers his sincerest gratitude to Professors Lennie Friedlander
and Mark Ashbaugh for comments and advice on drafts of this paper. The author also acknowledges the anonymous referee for insight on the route from (\ref{laptev-1}) to (\ref{new2}).

\end{document}